\documentclass{amsart}
\usepackage{graphicx}
\usepackage{float}
\usepackage{longtable}
\usepackage[margin=0pt,font=small,labelfont={normalsize},labelsep=space]{caption}
\newtheorem{thm}{Theorem}[section]
\newtheorem{cor}[thm]{Corollary}
\newtheorem{lem}[thm]{Lemma}
\newtheorem{prop}[thm]{Proposition}
\theoremstyle{definition}

\newtheorem*{theorem}{\textbf{Main Theorem}}
\newtheorem*{proof M}{\textbf{Proof of  Main Theorem}}
\theoremstyle{remark}

\numberwithin{equation}{section}

\begin{document}

\title{ On the Schur multiplier of nilpotent Lie superalgebra }%
\author[A. Shamsaki]{Afsaneh Shamsaki}
\address{School of Mathematics and Computer Science\\
Damghan University, Damghan, Iran}
\email{Shamsaki.Afsaneh@yahoo.com}
\author[P. Niroomand]{Peyman Niroomand}
\email{niroomand@du.ac.ir, p$\_$niroomand@yahoo.com}
\address{School of Mathematics and Computer Science\\
Damghan University, Damghan, Iran}


\begin{abstract}
Let $L$ be an $(m\vert n)$-dimensional nilpotent Lie superalgebra where $m + n \geq 4$ and $n \geq 1$. This paper classifies such nilpotent Lie superalgebras $L$ with a derived subsuperalgebra of dimension $m+n-2$ such that $\gamma(L) = m + 2n - 2 - \dim \mathcal{M}(L)$, where $\gamma(L) \in \{0, 1, 2\}$ and $\mathcal{M}(L)$ denotes the Schur multiplier of $L$. Furthermore, we show that all these superalgebras are capable.
\end{abstract}

\maketitle

\section{Introduction and Preliminaries}

Let $L$ be a finite-dimensional Lie superalgebra, and let $L \cong F/R$, where $F$ is a free Lie superalgebra. The Schur multiplier $\mathcal{M}(L)$ of $L$ is isomorphic to $R \cap F^2 / [R, F]$ as shown in \cite{Nay1}. Classifying Lie superalgebras of arbitrary dimensions without any constraints is a challenging task. In fact, the classification of all finite-dimensional Lie superalgebras remains an open problem. A common approach is to classify nilpotent Lie superalgebras based on a fixed invariant. Some researchers focus on the structure of $L$ by examining the dimension of the Schur multiplier. This invariant aids in classifying certain nilpotent Lie superalgebras. For example, the classification of nilpotent Lie superalgebras based on the dimension of the Schur multiplier is discussed in \cite{Nay2}. According to \cite{Nay1}, the dimension of the Schur multiplier of an $(m\mid n)$-dimensional nilpotent Lie superalgebra $L$ with the derived supersubalgebra of maximum dimension $m+n-2$ is bounded by $m+2n-2$. Thus, there exists a non-negative integer $\gamma(L)$ such that $\gamma(L) = m + 2n - 2 - \dim \mathcal{M}(L)$.

Throughout this paper, we assume that $L$ is a Lie superalgebra over a field of characteristic not equal to 2 or 3. We classify all finite-dimensional nilpotent Lie superalgebras $L$ when $\gamma(L) \in \{0, 1, 2\}$. Additionally, we demonstrate that all these Lie superalgebras are capable. Recall from \cite{Rud2} that a Lie superalgebra $L$ is said to be capable if there exists a Lie superalgebra $M$ such that $L \cong M/Z(M)$. Also, $Z^{*}(L)$ is the epicenter of a Lie superalgebra $L$, defined as the smallest ideal in $L$ such that $L/Z^{*}(L)$ is capable. Clearly, the ideal $Z^{*}(L)=0$ if and only if $L$ is capable. For more details on capable Lie superalgebras, see \cite{Rud2}.

The following theorem is a useful tool for identifying capable Lie superalgebras.

\begin{thm}\cite[Theorem 4.9]{Rud1}\label{6}
Let $N$ be a central ideal of the Lie superalgebra $L$. Then the following conditions are equivalent:
\begin{itemize}
    \item[(i)] $\mathcal{M}(L) \cong \mathcal{M}(L/N)/N \cap L^2,$ 
    \item[(ii)] $N \subseteq Z^{*}(L),$
    \item[(iii)] the natural map $\mathcal{M}(L) \longrightarrow \mathcal{M}(L/N)$ is a monomorphism.
\end{itemize}
\end{thm}

The lower central series is defined by $C^{0}(L) = L$ and inductively by $C^{i+1}(L) = [L, C^{i}(L)]$. A Lie superalgebra $L$ is nilpotent if there exists an integer $n$ such that $C^{n}(L) = 0$. Let $L = L_{\overline{0}} \oplus L_{\overline{1}}$. Similarly, we may define the following sequences for $L$ as defined in \cite{G}:

\begin{align*}
 &C^{0}(L_{\overline{0}}) = L_{\overline{0}} \quad \text{and} \quad C^{k}(L_{\overline{0}}) = [L_{\overline{0}}, C^{k-1}(L_{\overline{0}})] \quad \text{for all} \quad k \geq 1,\\
 &C^{0}(L_{\overline{1}}) = L_{\overline{1}} \quad \text{and} \quad C^{k}(L_{\overline{1}}) = [L_{\overline{0}}, C^{k-1}(L_{\overline{1}})] = ad_{L_{\overline{0}}} (C^{k-1}(L_{\overline{1}})) \quad \text{for all} \quad k \geq 1.
\end{align*}

\begin{thm}\cite[Theorem 2.1]{G}
The Lie superalgebra $L = L_{\overline{0}} \oplus L_{\overline{1}}$ is nilpotent if and only if there exist $(p, q) \in \mathbb{N} \times \mathbb{N}$ such that $C^{p-1}(L_{\overline{0}}) \neq 0$, $C^{q-1}(L_{\overline{1}}) \neq 0$, and $C^{p}(L_{\overline{0}}) = C^{q}(L_{\overline{1}}) = 0$.
\end{thm}

For convenience, we list some results regarding the Schur multiplier of a Lie superalgebra.

\begin{thm}\cite[Theorem 3.4]{Nay1}\label{th1.1}
Let $L$ be an $(m\vert n)$-dimensional Lie superalgebra. Then 
\begin{equation*}
\dim \mathcal{M}(L) = \frac{1}{2} \big( (m+n)^2 + (n-m) \big)
\end{equation*}
if and only if $L$ is abelian.
\end{thm}

\begin{thm}\cite[Theorem 3.7]{Nay1}\label{th1.2}
Let $L = L_{\overline{0}} \oplus L_{\overline{1}}$ be a Lie superalgebra with $\dim L = (m\vert n)$ and let $K \subseteq Z(L)$ be a graded ideal. If $H \cong L/K$ is the quotient Lie superalgebra, then 
\begin{align*}
&(i) \quad \dim \mathcal{M}(L) + \dim (L^2 \cap K) \leq \dim \mathcal{M}(H) + \dim \mathcal{M}(K) + \dim (H/H^2 \otimes K/K^2), \\
&(ii) \quad \dim \mathcal{M}(L) + \dim (L^2 \cap K) \leq \frac{1}{2} \big( (m+n)^2 + (n-m) \big).
\end{align*}
\end{thm}

\begin{thm}\cite[Theorem 5.4]{Nay1}\label{tb}
Let $L$ be a non-abelian nilpotent Lie superalgebra with $\dim L = (m\vert n)$, where $m + n \geq 3$ and $\dim L^2 = m + n - 2$. 
\begin{itemize}
\item[(i)] If $m + n = 3$, then $\dim \mathcal{M}(L)$ is equal to 1 or 2.
\item[(ii)] If $m + n \geq 4$ and $n \geq 1$, then $\dim \mathcal{M}(L) \leq m + 2n - 2$.
\end{itemize}
\end{thm}

\section{Main Result}
In this section, we classify the structure of all $(m\vert n)$-dimensional nilpotent Lie superalgebras with a derived subsuperalgebra of dimension $m+n-2$, for $m+n \geq 4$, $n \geq 1$, and $\gamma(L) \in {0,1,2}$. Furthermore, we demonstrate that all such Lie superalgebras are capable.

The following proposition provides the Schur multipliers of $(m\vert n)$-dimensional nilpotent Lie superalgebras with a derived subsuperalgebra of dimension $m+n-2$ for $4 \leq m+n \leq 5$ and $n \geq 1$.  
\begin{prop}\label{pr22}
Let $ L $ be an $ (m\vert n)$-dimensional nilpotent Lie superalgebra with  derived subalgebra of dimension $ m+n-2 $ such that  $ 4\leq m+n\leq 5$ and $ n\geq 1. $   Then we have the following  table.
 \begin{longtable}{cccc} 
 \caption{}\\
\multicolumn{3}{c}{}\\
\hline \multicolumn{1}{c}{\textsf{Name}} & \multicolumn{1}{c}{\textsf{Basis}}& \multicolumn{1}{c}{\textsf{Nonzero multiplication}} &   \multicolumn{1}{c}{\textbf{$\dim \mathcal{M}(L)$}}\\
\hline
\endhead
\hline \multicolumn{4}{r}{\small \itshape continued on the next page}
\endfoot
\endlastfoot

 $ (2 \vert 2)_{1} $&    $e_{1}, e_{2}, f_{1}, f_{2}$& $[f_{1}, f_{1}] =e_{1}, [f_{2}, f_{2}] =e_{2}$ & $ 1$\\
 \\
$ (2 \vert 2)_{4} $&    $e_{1}, e_{2}, f_{1}, f_{2}$& $[f_{1}, f_{2}] =e_{1}, [f_{2}, f_{2}] =e_{2}$ & $ 2$\\
\\
 $ (2 \vert 2)_{6} $&    $e_{1}, e_{2}, f_{1}, f_{2}$& $[e_{2}, f_{2}] =f_{1}, [f_{2}, f_{2}] =e_{1}$ & $2 $\\
 \\
$ (1 \vert 3)_{1} $&    $e_{1}, f_{1}, f_{2}, f_{3}$& $[e_{1}, f_{2}] =f_{1}, [e_{1}, f_{3}] =f_{2}$ & $ 3$\\
\\
$ (1\vert 4)_{7} $&$e_{1}, f_{1}, ..., f_{4}$&$[e_{1}, f_{2}] =f_{1}, [e_{1}, f_{3}] =f_{2},  [e_{1}, f_{4}] =f_{3}$  & $ 3 $\\
\\
$ (3\vert 2)_{5}$&$e_{1}, e_{2}, e_{3},f_{1}, f_{2}$&$[f_{1}, f_{1}] =e_{2}, [f_{1}, f_{2}] =e_{1},  [f_{2}, f_{2}] =e_{3}$  &$ 2$ \\
\\
$ (3\vert 2)_{13}$&$e_{1}, e_{2}, e_{3},f_{1}, f_{2}$&$[e_{1}, e_{2}] =e_{3}, [e_{1}, f_{2}] =f_{1},  [f_{1}, f_{2}] =e_{3}$ & $ 3 $\\&&$ [f_{2},f_{2}]=2e_{2}$ \\
\\
$ (2\vert 3)_{18} $ & $e_{1}, e_{2}, f_{1},f_{2}, f_{3}$& $[e_{1}, f_{3}] =f_{1}, [e_{2}, f_{2}] =f_{1},  [f_{2}, f_{2}] =2e_{1}$ & $ 2 $\\&&$ [f_{2},f_{3}]=-e_{2}$\\
\\
$ (2\vert 3)_{19} $ & $e_{1}, e_{2}, f_{1},f_{2}, f_{3}$& $[e_{1}, f_{3}] =f_{1}, [e_{2}, f_{2}] =f_{1},  [f_{2}, f_{3}] =-e_{1}$ & $ 3 $\\&&$ [f_{3},f_{3}]=2e_{2}$\\
\\
$ (2\vert 3)_{22} $ & $e_{1}, e_{2}, f_{1},f_{2}, f_{3}$& $[e_{1}, f_{2}] =f_{1}, [e_{1}, f_{3}] =f_{2},  [f_{3}, f_{3}] =e_{2}$ & $ 3 $\\
\\
$ (2\vert 3)_{23} $ & $e_{1}, e_{2}, f_{1},f_{2}, f_{3}$& $[e_{1}, f_{2}] =f_{1}, [e_{1}, f_{3}] =f_{2},  [f_{1}, f_{3}] =-e_{2}$ & $ 2 $\\&&$ [f_{2},f_{2}]=e_{2}$\\
  \hline
  \label{tab1}
\end{longtable}   
\end{prop}
 \begin{proof}
  Let $ L\cong (2\vert 3)_{22} =\langle e_{1}, e_{2}, f_{1},f_{2}, f_{3}\mid  [e_{1}, f_{2}] =f_{1}, [e_{1}, f_{3}] =f_{2},  [f_{3}, f_{3}] =e_{2} \rangle $.  By  using the method of  in \cite{Nay1}, we  compute  $ \dim \mathcal{M}((2\vert 3)_{22}) $. Start with
   \begin{align*}
& [e_{1},e_{2}]=s_{1},&&  [e_{1},f_{1}]=s_{2}, &   & [e_{1},f_{2}]=f_{1}+s_{3},    \cr
  &[e_{1},f_{3}]=f_{2}+s_{4},&& [e_{2},f_{1}]=s_{5},&& [e_{2},f_{1}]=s_{6},\cr
    &     [e_{2},f_{3}]=s_{7},    &  &[f_{1},f_{2}]=s_{8},&& [f_{1},f_{3}]=s_{9}, \cr
    &[f_{1},f_{1}]=s_{10},  && [f_{2},f_{3}]=s_{11},  &&[f_{2},f_{2}]=s_{12},\cr
  &[f_{3},f_{3}]=e_2+s_{13},   && &&
\end{align*}
where the set $ \lbrace s_{1},..., s_{13} \rbrace$ generates $ \mathcal{M}(L) $. The graded Jacobi identity on all possible triples gives
\begin{align*}
&s_{1}=[e_{1},[f_{3},f_{3}]]=[f_{3},[e_{1},f_{3}]]-[f_{3},[f_{3},e_{1}]]=2s_{11},\cr
&s_{5}=[e_{2},[e_{1},f_{2}]]=-[f_{2},[e_{2},e_{1}]]-[e_{1},[f_{2},e_{2}]]=0,\cr
&s_{6}=[f_{3},[e_{2},e_{1}]]=[e_{1},[f_{3},e_{2}]]+[x_{2},[x_{5},x_{1}]]=0,\cr
&s_{7}=[f_{3},[f_{3},f_{3}]]=-[f_{3},[f_{3},f_{3}]]-[f_{3},[f_{3},f_{3}]]=-2s_7,\cr
&s_{8}=[f_{1},[e_{1},f_{3}]]=[f_{3},[f_{1},e_{1}]]+[e_{1},[f_{3},f_{1}]]=0,\cr
&s_{9}=[f_{1},[e_{1},f_{3}]]=[f_{3},[f_{1},e_{1}]]+[e_{1},[f_{3},f_{1}]]=0,\cr
&s_{10}=[f_{1},[e_{1},f_{2}]]=[f_{2},[f_{1},e_{1}]]+[e_{1},[f_{2},f_{1}]]=0.
\end{align*}

Let $f^{\prime}{1}=f{1}+s_{3}$, $f^{\prime}{2}=f{2}+s_{4}$, and $e^{\prime}{2}=e{2}+s_{13}$.
A change of variables allows $ s_{3}=s_4=s_{13}=0 $. Hence $ \mathcal{M}((2\vert 3){22}) =\langle s{1}, s_{2}, s_{9} \rangle$, and thus $\dim \mathcal{M}((2\vert 3)_{22})=3$.

By a similar method, the Schur multiplier of other Lie superalgebras in Table \ref{tab1} is obtained.
\end{proof}

In the remainder of the paper, we consider all $ (m\vert n) $-dimensional nilpotent Lie superalgebras $ L $ with the derived subsuperalgebra of dimension $ m+n-2$ such that $ m+n\geq 4$ and $ n\geq 1. $

Here's the revised draft:

```latex
\begin{lem}\label{lem1}
Let $ L $ be an $(m\vert n)$-dimensional nilpotent Lie superalgebra with the Schur multiplier of dimension $m+2n-2-i$ for all $i\geq 0$, and let $K\subseteq L^{2}\cap Z(L)$ be an $(1,0)$-dimensional ideal. Then $\dim (L/ K)^{2}=m+n-3$, and
\begin{itemize}
\item[(i)] if $m+n=4$, then $\dim \mathcal{M}(L/K)$ is equal to $1$ or $2$,
\item[(ii)] if $m+n\geq 5$, then $m+2n-2-(i+1)\leq \dim \mathcal{M}(L/K)\leq m+2n-3$ for all $i\geq 0$.
\end{itemize}
\end{lem}
\begin{proof}
Assume that $\dim L^{2}=m+n-2$ and $\dim L^{2} \cap K=1$. Thus, $\dim (L/ K)^{2}=m+n-3$.\\
(i) Since $\dim K=(1,0)$, we have $\dim L/K=(m-1, n)$ such that $m+n=3$, and by part (i) of Theorem \ref{tb}, $\dim \mathcal{M}(L/K)=1$ or $2$.\\
(ii) We know that $\dim \mathcal{M}(K)=0$ by Theorem \ref{th1.1}. It implies that $m+2n-(i+1)\leq \dim \mathcal{M}(L/K)$ for all $i\geq 0$, by Theorem \ref{th1.2}. On the other hand, $\dim \mathcal{M}(L/K)\leq m+2n-3$ by part (ii) of Theorem \ref{tb}. Hence the result follows.
\end{proof}
\begin{lem}\label{lem2}
Let $ L $ be an $(m\vert n)$-dimensional nilpotent Lie superalgebra with the Schur multiplier of dimension $m+2n-2-i$ for all $i\geq 0$, and let $K\subseteq L^{2}\cap Z(L)$ be an $(0,1)$-dimensional ideal. Then $\dim (L/ K)^{2}=m+n-3$, and
\begin{itemize}
\item[(i)] if $m+n=4$, then $\dim \mathcal{M}(L/K)$ is equal to $1$ or $2$,
\item[(ii)] if $m+n\geq 5$ and $n=1$, then $\dim \mathcal{M}(L/K)\leq m-2$,
\item[(iii)] if $m+n\geq 5$ and $n\geq 2$, then $m+2n-2-(i+2)\leq \dim \mathcal{M}(L/K) \leq m+2n-4$ for all $i\geq 0$.
\end{itemize}
\end{lem}
\begin{proof}
The parts (i) and (iii) are obtained by a similar method to Lemma \ref{lem1}. It is sufficient to prove part (ii).\\
(ii) Since $m+n\geq 5$ and $n=1$, we have $\dim (L/ K)=(m,0)$. So, $L/K$ is a Lie algebra, and $\dim \mathcal{M}(L/K)\leq m-2$ by \cite[Theorem 3.1]{Nir}.
\end{proof}
\begin{thm}\label{th2.3}
Let $ L $ be an $(m\vert n)$-dimensional nilpotent Lie superalgebra. Then $\dim \mathcal{M}(L)< m+2n-3$. 
\end{thm}
\begin{proof}
We proceed by induction on $m+n$. Let $m+n=4$. By looking at Table \ref{tab1}, there is no $(m\vert n)$-Lie superalgebra such that $m+n=4$ and $n\geq 1$ with $\dim \mathcal{M}(L)<m+2n-2$. Let $m+n\geq 5$ and $K$ be a central ideal of $L$ such that $\dim K=1$ and $K\subseteq Z(L)\cap L^{2}$. By Theorem \ref{th1.2}, we have
\begin{align*}
1+\dim (\mathcal{M}(L))\leq \dim (\mathcal{M}(L/K))+\dim (\mathcal{M}(K))+\dim (L/L^{2}\otimes K)&
\end{align*}
If $K\subseteq L_{\overline{0}}$, then $\dim (\mathcal{M}(K))=0$ by Theorem \ref{th1.1}. Also, we have $\dim (\mathcal{M}(L/K))<m+2n-3$ by the induction hypothesis. So, $\dim \mathcal{M}(L)< (m+2n-3)-1+\dim L/L^{2}$ by Theorem \ref{th1.2}.\\
If $K\subseteq L_{\overline{1}}$, then $\dim \mathcal{M}(K)=1$ by Theorem \ref{th1.1}. Hence
\begin{align*}
\dim \mathcal{M}(L)< (m+2n-4)+\dim L/L^{2}
\end{align*}
Thus $\dim \mathcal{M}(L)< m+2n-2$. Suppose that $\dim \mathcal{M}(L)= m+2n-3$. By a similar method, we can see $\dim \mathcal{M}(L)< m+2n-3$. The result follows.
\end{proof}

Here's the revised version:

\begin{prop}\label{pr24}
No $(m\vert n)$-dimensional nilpotent Lie superalgebras exist with a derived subsuperalgebra of dimension $4$, and $\dim \mathcal{M}(L) = m + 2n - 4$ when $m + n = 6$ and $n \geq 1$.
\end{prop}
\begin{proof}
Given $m + n = 6$ and $n \geq 1$, the superalgebra $(m\vert n)$ can be one of the following: $(5\vert 1)$, $(4\vert 2)$, $(3\vert 3)$, $(2\vert 4)$, $(1\vert 5)$, or $(0\vert 6)$. Assuming $K \subseteq L^{2} \cap Z(L)$ and $(m\vert n) = (4\vert 2)$, if $K \subseteq L_{\overline{0}}$, then $\dim L/K = (3\vert 2)$. By Lemma \ref{lem1}, $\dim (L/K)^{2} = 3$, and $\dim \mathcal{M}(L/K) \geq 3$. Thus, $L/K$ is isomorphic to a nilpotent Lie superalgebra of superdimension $(3\vert 2)_{13}$ as per Table \ref{tab1}. Let $K=\langle e_4 \rangle$ and 
\[ (3\vert 2)_{13} = \langle e_{1}, e_2, e_{3}, f_1, f_2\vert [e_{1}, e_{2}] =e_{3}, [e_{1}, f_{2}] =f_{1},  [f_{1}, f_{2}] =e_{3}, [f_{2}, f_{2}] =2e_{2}\rangle. 
\]
We obtain $L=\langle e_{1}, e_2, e_{3}, e_4, f_1, f_2 \rangle$ with brackets:
\begin{align*}
  &[e_{1},e_{2}]=e_{3}+\alpha_1 e_{4},   ~\quad     [e_{1},e_{3}]=\alpha_2 e_{4}, \quad  [e_{1},f_{1}]=0,    \cr
 &[e_{1},f_{2}]=f_{1}, ~\qquad\qquad [e_{2},e_{3}]=\alpha_3 e_{4},\quad [e_{2},f_{1}]=0,\cr
   &    [e_{2},f_{2}]=0, \qquad \quad\qquad  [e_{3}, f_{1}]=0, ~\qquad [e_{3}, f_{2}]=0, \cr
    &[f_{1}, f_{2}]=e_3+\alpha_{4} e_4,   ~\quad  [f_{1}, f_{1}]=\alpha_{5} e_4, \quad [f_{2}, f_{2}]=2e_2+\alpha_{6} e_4.
\end{align*}
Let $e^{\prime}_{3}=e_{3}+\alpha_1 e_{4}$ and  $e^{\prime}_{2}=e_2+1/2\alpha_{6} e_4$. A change of variables sets $\alpha_1=\alpha_6=0$. By graded Jacobi identities, we have
\begin{align*}
&\alpha_2 e_{4}=[e_{1},e_{3}]=[e_{1},[f_{1},f_{2}]]=[f_{2},[e_{1},f_{1}]]-[f_{1},[f_{2},e_{1}]]=\alpha_5 e_{4},\cr
&\alpha_3 e_{4}=[e_{2},e_{3}]=[e_{2},[f_{1},f_{2}]]=[f_{2},[e_{2},f_{1}]]-[f_{1},[f_{2},e_{2}]]=0.
\end{align*}
Thus, 
\[ L=\langle e_{1}, e_2, e_{3},e_{4}, f_1, f_2 \vert [e_{1}, e_{2}] =e_{3}, [e_{1}, e_{3}] =[f_1, f_1]=\alpha_{2}e_{4},  [e_{1}, f_{2}] =f_{1},
\]
 \[[f_1, f_{2}] =e_3+\alpha_4e_4,    [f_{2}, f_{2}] =2 e_{2}\rangle.
\]
Since $\dim L^{2}=4$, $\alpha_2$ or $\alpha_4$ must not be zero. By a similar method as in Proposition \ref{pr22}, $\dim \mathcal{M}(L)\leq 3$. Thus, there's no such Lie superalgebra of dimension $(4 \vert 2)$ with $\dim \mathcal{M}(L)=4$. If $K\subseteq L_{\overline{1}}$, then $\dim L/K=(4\vert 1)$. By Table \ref{tab1}, there's no such Lie algebra.

For $(m\vert n) = (5\vert 1)$, $(3\vert 3)$, $(2\vert 4)$, $(1\vert 5)$, or $(0\vert 6)$, there's no such nilpotent Lie superalgebra $L$ with derived subsuperalgebra of dimension $4$ such that $m+n=6$, $n\geq 1$, and $\dim \mathcal{M}(L)=m+2n-4$ by a similar method.   
 \end{proof}

\begin{thm}\label{th2.4}
For an $(m\vert n)$-dimensional nilpotent Lie superalgebra with $m+n\geq 6$ and $n\geq 1$, $\dim \mathcal{M}(L)< m+2n-4$. 
\end{thm}
\begin{proof}
The proof follows similar methods as in the proof of Theorem \ref{th2.3}, Proposition \ref{pr24}, Theorems \ref{tab1} and \ref{th1.2}.  
\end{proof}

We now derive the structure of all nilpotent Lie superalgebras of dimension $(m\vert n)$ when $\gamma(L)\in \lbrace 0, 1, 2\rbrace$.

\begin{theorem}\label{lt}
Let $ L $ be an $(m\vert n)$-dimensional nilpotent Lie superalgebra. 

\begin{itemize}
\item[(i)] There is no such Lie superalgebra with $ \gamma(L)=0. $
\item[(ii)] There is no such Lie superalgebra with $ \gamma(L)=1. $
\item[(iii)]  $ \gamma(L)=2 $ if and only if  $ L $ is isomorphic to one of the Lie superalgebras $ (2\vert 2)_{4}, $ $ (2\vert 2)_{6}, $ $ (1\vert 3)_{1}, $ $ (3\vert 2)_{13} $ or $ (2\vert 3)_{18}. $
\end{itemize}
\end{theorem}
\begin{proof}
Suppose $ \gamma(L)=0 $ or $ 1. $ Then $ \dim \mathcal{M}(L)=m+2n-2$ or $ m+2n-3, $ respectively. Therefore, by Theorem \ref{th2.3}, there exists no $(m\vert n)$-dimensional nilpotent Lie superalgebra with $ 4\leq m+n\leq 5 $ and $ n\geq 1 $. 

For part (iii), using Theorem \ref{th2.4}, there is no such nilpotent Lie superalgebra $ L $ with $ \dim \mathcal{M}(L)=m+2n-4$ and $ \dim L^{2}=m+n-2 $ for $ m+n\geq 6. $ Hence, for every nilpotent Lie superalgebra with $ \dim \mathcal{M}(L)=m+2n-4$ and $ \dim L^{2}=m+n-2 $, we have $ m+n\leq 5. $ By examining Table \ref{tab1}, $ L $ should be isomorphic to one of the Lie superalgebras $ (2\vert 2)_{4}, $ $ (2\vert 2)_{6}, $ $ (1\vert 3)_{1}, $ $ (3\vert 2)_{13} $ or $ (2\vert 3)_{18}. $ The converse is obtained by referring to Table \ref{tab1}.
\end{proof}
\begin{cor}
Let $ L $ be an $(m\vert n)$-dimensional nilpotent Lie superalgebra with $ n+m \geq6 $. Then $$ \dim \mathcal{M}(L)\leq m+2n-5.$$
\end{cor}
\begin{cor}
Let $ L $ be an $(m\vert n)$-dimensional nilpotent Lie superalgebra. If $\gamma(L)=2, $ then $ L $ is capable.
\begin{proof}
Let $ L $ be an $(m\vert n)$-dimensional nilpotent Lie superalgebra with the derived supersubalgebra of dimension $ m+n-2 $ and $ \dim \mathcal{M}(L)=m+2n-4.$ Then $ \dim \mathcal{M}(L/K)\leq \dim \mathcal{M}(L) $ by Lemmas \ref{lem1} and \ref{lem2}. Hence, $ (2\vert 2)_{4}, $ $ (2\vert 2)_{6}, $ $ (1\vert 3)_{1}, $ $ (3\vert 2)_{13} $ and  $ (2\vert 3)_{18} $ are capable by Theorem \ref{6}. The result follows. 
\end{proof}
\end{cor}

\end{document}